\documentclass[12pt]{article}
\usepackage{amssymb}
\usepackage{amsmath}
\usepackage{amsthm}
\usepackage[colorlinks=true,allcolors=black]{hyperref}
\usepackage{float}
\usepackage{tikz}
\tikzset{every picture/.style={line width=0.75pt}}
\usepackage[labelfont=bf]{caption}

\newtheorem{theorem}{Theorem}

\newtheorem{definition}{Definition}
\newtheorem{remark}{Remark}

\newtheorem{corollary}{Corollary}

\NewDocumentCommand{\sff}{}{\mathrm{I\!I}}

\newcommand{\Addresses}{{
  \bigskip
  \footnotesize

KMS:
\textsc{Department of Mathematics, Johns Hopkins University, 404 Krieger Hall, 3400 North Charles Street, Baltimore, MD 21218, USA}\par\nopagebreak
  \textit{Email address:} \texttt{kmarsh34@jh.edu}

\bigskip

GN:
\textsc{Department of Mathematics, University of Rochester, 915 Hylan Building, Rochester, NY 14620, USA}\par\nopagebreak
  \textit{Email address:} \texttt{g.niu@rochester.edu}

}}
\begin{document}
\title{\large \textbf{ISOPERIMETRY BY STRETCHING}}
\author{\small KOBE MARSHALL-STEVENS \& GONGPING NIU}
\date{\vspace{-5ex}}
\maketitle

\begin{abstract}
    \noindent We construct isoperimetric regions from separating hypersurfaces in closed manifolds. This yields isoperimetric boundaries exhibiting a wide variety of topological types and singular sets.
\end{abstract} 

\tableofcontents

\section{Introduction}\label{sec: introduction}

Isoperimetric boundaries minimise area for a fixed enclosed volume, with sharp regularity theory ensuring smoothness away from a closed singular set of codimension seven. While in closed Riemannian manifolds their existence is guaranteed for every proportion of enclosed volume, relatively few explicit examples are known; we refer to the survey in \cite[Appendix H]{EM13} and references therein. Here, we show that any smooth, closed, connected, separating hypersurface can be realised as an isoperimetric boundary for some Riemannian metric on a closed manifold (see Theorem \ref{thm: topological constructions}). This follows as an application of a `stretching' procedure for such hypersurfaces that are a priori assumed to be locally uniquely area-minimising, enabling a local-to-global transfer of isoperimetry when the ambient manifold is made sufficiently long (see Theorem \ref{thm: stretching}). 

\bigskip

As our technique requires local minimisation, and crucially not smoothness, of the underlying hypersurface, it can be applied to verify a claim made in \cite{M16} and exhibit many singular isoperimetric boundaries in Riemannian spheres (see Subsection \ref{subsec: singular sets}). Moreover, with the construction of singular minimal hypersurfaces in \cite{S00}, we show that a wide variety of higher dimensional singular sets can arise in isoperimetric boundaries (see Theorem \ref{thm: singular constructions}); this includes any finite collection of smooth, closed, oriented manifolds of appropriate codimension. To the best of our knowledge, the only singular examples of isoperimetric boundaries prior to this work were those of \cite{N24}. In that work, the singular sets consisted of two isolated points and, due to a topological restriction, were modelled on symmetric Simons cones (namely cones over products of spheres of the same dimension). In contrast, our construction allows for singular sets modelled on any cylindrical tangent cone, provided that the cone factor is strictly stable and strictly minimising (see Remark \ref{rem: cones}).

\subsection{Strategy}

We now briefly outline the construction that yields Theorem \ref{thm: stretching}. Suppose that we are given a smooth, closed, connected, separating hypersurface, $\Sigma$, which is locally uniquely area-minimising in a closed Riemannian manifold, $(M,g)$. The bounded geometry of $(M,g)$ ensures that any isoperimetric region enclosing the same volume as $\Sigma$ has uniform bounds on both the diameter and number of the connected components of its boundary. By elongating portions of the tubular neighbourhood of $\Sigma$ appropriately, we obtain a family of Riemannian metrics for which the above-mentioned uniform bounds persist for any isoperimetric region with the same enclosed volume as $\Sigma$. When these manifolds are sufficiently elongated, we are able to utilise the uniform bounds on the boundary components, the unique local area-minimisation of $\Sigma$, and an elementary volume comparison argument to ensure that $\Sigma$ must be the unique isoperimetric boundary for its enclosed volume.
To carry the same construction through in the singular setting, one needs to ensure some topological control on the neighbourhood in which $\Sigma$ is uniquely minimising (see Corollary \ref{cor: singular stretching}). As a by-product of the proof of Theorem \ref{thm: stretching}, the construction above also produces isoperimetric regions from volume-constrained minimisers that are, a priori, assumed to be sufficiently close in area to unique local area-minimisers (see Corollary \ref{cor: stretching for vcm}).

\begin{figure}[H]
    \centering
\captionsetup{justification=justified,margin=1cm}
\begin{tikzpicture}[x=0.75pt,y=0.75pt,yscale=-1.25,xscale=1.25]

\draw    (239.5,149.75) .. controls (251.32,160.5) and (266.37,160.5) .. (277.62,149.75) ;
\draw    (239.88,190.25) .. controls (251.51,178.5) and (266.37,178.5) .. (278,190.25) ;
\draw    (239.88,190.25) .. controls (226.38,203.5) and (204.5,193.5) .. (205,169.5) .. controls (205.5,145.5) and (226.5,136) .. (239.5,149.75) ;
\draw    (277.62,149.75) .. controls (291.13,136.51) and (313,146.52) .. (312.49,170.52) .. controls (311.97,194.52) and (290.96,204.01) .. (277.97,190.25) ;
\draw    (321.5,169.75) -- (357.5,169.98) ;
\draw [shift={(360.5,170)}, rotate = 180.37] [fill={rgb, 255:red, 0; green, 0; blue, 0 }  ][line width=0.08]  [draw opacity=0] (8.93,-4.29) -- (0,0) -- (8.93,4.29) -- cycle    ;
\draw    (259.43,158) .. controls (261.57,166.71) and (261.86,173) .. (259.57,181.57) ;
\draw  [dash pattern={on 4.5pt off 4.5pt}]  (259.43,158) .. controls (259,164.14) and (255.86,172.14) .. (259.57,181.57) ;
\draw    (406.33,189.42) .. controls (392.83,202.67) and (370.33,192.67) .. (370.83,168.67) .. controls (371.33,144.67) and (392.33,135.17) .. (405.33,148.92) ;
\draw  [dash pattern={on 4.5pt off 4.5pt}]  (497.43,157.8) .. controls (497.27,160.1) and (496.73,162.65) .. (496.31,165.45) .. controls (495.61,170.12) and (495.25,175.47) .. (497.57,181.37) ;
\draw  [draw opacity=0] (214.99,171.98) .. controls (216.9,169.99) and (219.49,168.77) .. (222.33,168.79) .. controls (225.15,168.81) and (227.69,170.04) .. (229.53,172.03) -- (222.15,180.1) -- cycle ; \draw   (214.99,171.98) .. controls (216.9,169.99) and (219.49,168.77) .. (222.33,168.79) .. controls (225.15,168.81) and (227.69,170.04) .. (229.53,172.03) ;  
\draw  [draw opacity=0] (231.95,168.73) .. controls (230.4,171.43) and (226.84,173.43) .. (222.63,173.65) .. controls (217.67,173.91) and (213.43,171.61) .. (212.27,168.3) -- (222.37,166.11) -- cycle ; \draw   (231.95,168.73) .. controls (230.4,171.43) and (226.84,173.43) .. (222.63,173.65) .. controls (217.67,173.91) and (213.43,171.61) .. (212.27,168.3) ;  
\draw    (275.43,151.22) .. controls (277.57,164.61) and (277.86,174.27) .. (275.57,187.44) ;
\draw    (242.54,152.33) .. controls (244.68,165.31) and (244.97,174.68) .. (242.68,187.44) ;
\draw  [draw opacity=0] (286.32,172.31) .. controls (288.23,170.32) and (290.82,169.1) .. (293.66,169.12) .. controls (296.48,169.14) and (299.02,170.38) .. (300.86,172.37) -- (293.48,180.44) -- cycle ; \draw   (286.32,172.31) .. controls (288.23,170.32) and (290.82,169.1) .. (293.66,169.12) .. controls (296.48,169.14) and (299.02,170.38) .. (300.86,172.37) ;  
\draw  [draw opacity=0] (303.28,169.07) .. controls (301.73,171.76) and (298.17,173.76) .. (293.97,173.98) .. controls (289,174.25) and (284.76,171.95) .. (283.61,168.63) -- (293.71,166.44) -- cycle ; \draw   (303.28,169.07) .. controls (301.73,171.76) and (298.17,173.76) .. (293.97,173.98) .. controls (289,174.25) and (284.76,171.95) .. (283.61,168.63) ;  
\draw    (411.87,151.33) .. controls (414.02,164.31) and (414.3,173.68) .. (412.02,186.44) ;
\draw  [draw opacity=0] (383.32,172.31) .. controls (385.23,170.32) and (387.82,169.1) .. (390.66,169.12) .. controls (393.48,169.14) and (396.02,170.38) .. (397.86,172.37) -- (390.48,180.44) -- cycle ; \draw   (383.32,172.31) .. controls (385.23,170.32) and (387.82,169.1) .. (390.66,169.12) .. controls (393.48,169.14) and (396.02,170.38) .. (397.86,172.37) ;  
\draw  [draw opacity=0] (400.28,169.07) .. controls (398.73,171.76) and (395.17,173.76) .. (390.97,173.98) .. controls (386,174.25) and (381.76,171.95) .. (380.61,168.63) -- (390.71,166.44) -- cycle ; \draw   (400.28,169.07) .. controls (398.73,171.76) and (395.17,173.76) .. (390.97,173.98) .. controls (386,174.25) and (381.76,171.95) .. (380.61,168.63) ;  
\draw    (477.5,149.55) .. controls (489.32,160.3) and (504.37,160.3) .. (515.62,149.55) ;
\draw    (477.88,190.05) .. controls (489.51,178.3) and (504.37,178.3) .. (516,190.05) ;
\draw    (497.43,157.8) .. controls (499.57,166.51) and (499.86,172.8) .. (497.57,181.37) ;
\draw    (405.33,148.92) .. controls (408.78,152) and (407.41,151.33) .. (411.87,151.33) .. controls (416.33,151.33) and (414.45,145.27) .. (429,145.64) .. controls (443.55,146) and (473,144) .. (477.5,149.55) ;
\draw    (406.33,189.42) .. controls (409.78,186.33) and (408.41,187) .. (412.87,187) .. controls (417.33,187) and (415.45,193.06) .. (430,192.7) .. controls (444.55,192.33) and (473.22,194.44) .. (477.88,190.05) ;
\draw    (586.8,189.42) .. controls (600.3,202.67) and (622.8,192.67) .. (622.3,168.67) .. controls (621.8,144.67) and (600.8,135.17) .. (587.8,148.92) ;
\draw  [dash pattern={on 4.5pt off 4.5pt}]  (581.26,151.33) .. controls (579.12,164.31) and (578.83,173.68) .. (581.12,186.44) ;
\draw    (581.26,151.33) .. controls (584,160.5) and (584.83,172.4) .. (581.12,186.44) ;
\draw  [draw opacity=0] (609.82,172.31) .. controls (607.9,170.32) and (605.31,169.1) .. (602.47,169.12) .. controls (599.65,169.14) and (597.11,170.38) .. (595.27,172.37) -- (602.65,180.44) -- cycle ; \draw   (609.82,172.31) .. controls (607.9,170.32) and (605.31,169.1) .. (602.47,169.12) .. controls (599.65,169.14) and (597.11,170.38) .. (595.27,172.37) ;  
\draw  [draw opacity=0] (592.85,169.07) .. controls (594.4,171.76) and (597.96,173.76) .. (602.17,173.98) .. controls (607.13,174.25) and (611.37,171.95) .. (612.53,168.63) -- (602.43,166.44) -- cycle ; \draw   (592.85,169.07) .. controls (594.4,171.76) and (597.96,173.76) .. (602.17,173.98) .. controls (607.13,174.25) and (611.37,171.95) .. (612.53,168.63) ;  
\draw    (587.8,148.92) .. controls (584.36,152) and (585.72,151.33) .. (581.26,151.33) .. controls (576.8,151.33) and (578.68,145.27) .. (564.13,145.64) .. controls (549.59,146) and (520.13,144) .. (515.63,149.55) ;
\draw    (586.8,189.42) .. controls (583.36,186.33) and (584.72,187) .. (580.26,187) .. controls (575.8,187) and (577.68,193.06) .. (563.13,192.7) .. controls (548.59,192.33) and (520.66,194.44) .. (516,190.05) ;
\draw  [dash pattern={on 4.5pt off 4.5pt}]  (242.54,152.33) .. controls (240.4,165.31) and (240.11,174.68) .. (242.4,187.44) ;
\draw  [dash pattern={on 4.5pt off 4.5pt}]  (411.87,151.33) .. controls (409.73,164.31) and (409.44,173.68) .. (411.73,186.44) ;
\draw  [dash pattern={on 4.5pt off 4.5pt}]  (275.71,152.33) .. controls (273.57,165.31) and (273.29,174.68) .. (275.57,187.44) ;

\end{tikzpicture}
    \caption{In each of the two graphics above, the three circles depict $\Sigma$ between the boundaries of the tubular neighbourhood in which it is uniquely area-minimising. Along the stretching procedure, the `caps' outside of this tubular neighbourhood are unchanged.}
    \label{fig: stretching}
\end{figure}
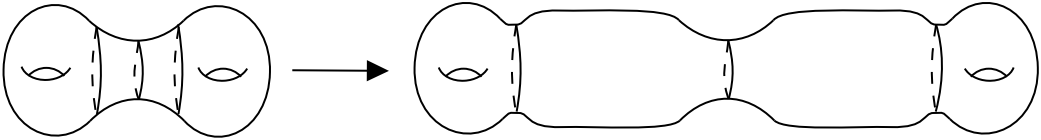

We now proceed as follows: Section \ref{sec: stretching} records the stretching of Theorem \ref{thm: stretching} and the above mentioned Corollaries \ref{cor: singular stretching} and \ref{cor: stretching for vcm}. Section \ref{sec: applications} contains the applications of Theorem \ref{thm: stretching}, with Subsection \ref{subsec: topology} devoted to the smooth setting and Subsection \ref{subsec: singular sets} to the construction of singular isoperimetric boundaries.

\subsection{Notation}\label{subsec: notation}

We now collect some notation and definitions that will be used throughout this work:

\begin{itemize}
    \item Throughout, we let $M$ denote a smooth, closed (i.e.~compact with empty boundary) manifold of dimension $n + 1 \geq 2$. Without loss of generality we will assume that $M$ is connected. We say that a set, $S \subset M$, is \textbf{separating} if there are disjoint, connected, open sets, $E,F \subset M$, with $S = \partial E = \partial F$ and $M = E \cup S \cup F$. By a \textbf{hypersurface}, $\Sigma \subset M$, we will mean a closed subset $\Sigma = \mathrm{Reg}(\Sigma) \cup \mathrm{Sing}(\Sigma)$, where the \textbf{regular set}, $\mathrm{Reg}(\Sigma)$, is a smoothly embedded manifold of dimension $n$, and the \textbf{singular set}, $\mathrm{Sing}(\Sigma)$, is a (possibly empty) closed set of Hausdorff dimension at most $n-7$. Such hypersurfaces will arise as boundaries of Caccioppoli sets (defined below), which we may regard as integral $n$-currents with $\mathbb{Z}$ coefficients when the regular part is oriented, or $n$-chains with $\mathbb{Z}_2$ coefficients otherwise.

    \item Given a smooth Riemannian metric, $g$, on $M$ we denote $\mathrm{dist}_g$ for the distance function on $M$, $B^g_r(p)$ the open geodesic ball in $M$ of radius $r > 0$ centred at $p \in M$, $\mathrm{Vol}_g(E)$ for the volume of a measurable set $E \subset M$, and $\int_E u \, dV_g$ for the integral of a measurable function $u$, all with respect to the metric $g$. When working in $\mathbb{R}^{n+1}$ we will write $g_\mathrm{eucl}$ for the Euclidean metric. We may omit metric dependence from our notation when working in Euclidean space or when it is clear from context.

    \item A measurable set $E \subset M$ is a \textbf{Caccioppoli set} if the indicator function of $E$ is of bounded variation, or equivalently if it has finite \textbf{perimeter} with respect to the metric $g$ so that
    \begin{equation*}
    \mathrm{Per}_g(E) = \sup \left\{\int_E \mathrm{div}_g X \, dV_g \: \bigg| \: X \in \Gamma(TM), g(X,X) \leq 1 \right\} < \infty,
    \end{equation*}
    where $\mathrm{div}_g$ is the divergence with respect to the metric $g$ and $\Gamma(TM)$ is the set of vector fields on $M$; we write $\mathrm{Per}_g(E;U)$ for the perimeter relative to a measurable set $U \subset M$ by restricting the integral above to $U$. We write $\mathcal{C}(S)$ for the set of Caccioppoli sets in a set $S$, and note that $\mathcal{C}(S)$ is independent of the choice of metric on $S$. For a Caccioppoli set, $U \in \mathcal{C}(M)$, we denote by $[U]$ the associated integral $(n+1)$-current with coefficients in $\mathbb Z_2$; then $\partial U =\mathrm{Spt}(\partial[U])$.
    
    \item We say that $\Omega \in \mathcal{C}(M)$ is an \textbf{isoperimetric region} of enclosed volume $t \in \mathbb{R}$ if
    \begin{equation*}
        \mathrm{Per}_g(\Omega ) = \inf_{E \in \mathcal{C}(M)}\{\mathrm{Per}_g(E) \: | \: \mathrm{Vol}_g(E) = t\}.
    \end{equation*}
    The existence of isoperimetric regions of a given enclosed volume in $M$ is then guaranteed by the direct method of the calculus of variations (e.g.~see \cite[Section 13.2]{M16}). Moreover, we say that $\Omega \in \mathcal{C}(M)$ is a \textbf{volume-constrained minimiser} in an open set $U \subset M$ if 
    \begin{align*}
        \mathrm{Per}_g(\Omega; U ) = \inf_{E\in \mathcal{C}(M)} \{\mathrm{Per}_g(E; U) \, | \, \Omega \Delta E \subset \subset U \text{ and  } \mathrm{Vol}_g(\Omega\cap U) = \mathrm{Vol}_g(E\cap U) \}.
    \end{align*}
    
    \item Given an isoperimetric region $\Omega$ as above, we denote by $\Sigma =\partial \Omega$ the closed, two-sided, embedded hypersurface of constant mean curvature (defined on its regular part); by \cite{GMT83} and \cite{M03}, we ensure that $\mathrm{Sing}(\Sigma)$ is of Hausdorff dimension at most $n-7$. We write $|\Sigma|_g$ for the area of $\Sigma$ and denote by $\nu_g$ the outward pointing unit normal to $\Sigma$ with respect to the metric $g$.

     \item Suppose that $\mathbf{C} \subset \mathbb{R}^{n+1}$ is a \textbf{regular} minimal hypercone, so that $\mathrm{Sing}(\mathbf{C}) \subset \{0\}$. We then denote the smooth, closed, manifold $\Sigma = \mathbf{C} \cap \mathbb{S}^n \subset \mathbb{R}^{n+1}$ as its \textbf{link}. By parameterising $\mathbf{C}$ in radial coordinates, $(r,\omega) \in (0,\infty) \times \Sigma$, we decompose (as in \cite{S68}) the Jacobi operator, $L_{\mathbf{C}}$, of $\mathbf{C}$ as
    $$L_{\mathbf{C}} = \partial^2_r + \left(\frac{n-1}{r}\right)\partial_r + \frac{1}{r^2}(\Delta_\Sigma + |\mathrm{I\!I}_\Sigma|^2),$$
    where $\mathrm{I\!I}_\Sigma$ is the second fundamental form of $\Sigma$ in $\mathbb{S}^n$. We let $\mu_1 < \mu_2 \leq \mu_3 \leq \ldots \nearrow + \infty$ be the eigenvalues of $- (\Delta_\Sigma + \vert \mathrm{I\!I}_\Sigma \vert^2)$, and $\varphi_1, \varphi_2, \ldots$ be the corresponding $L^2(\Sigma)$-orthonormal eigenfunctions where $\varphi_1 > 0$. By \cite{S68}$, \mathbf{C}$ is \textbf{stable} if and only if $\mu_1\geq -\frac{(n-2)^2}{4}$ and, we say (as in \cite{CHS84}) that $\mathbf{C}$ is \textbf{strictly stable} if $\mu_1 > -\frac{(n-2)^2}{4}$; the relationships above regarding inequalities for $\mu_1$ and the stability of $\mathbf{C}$ are made explicit in \cite[Section 4]{CHS84}. Moreover, we say (as in \cite{HS85}) that a minimal hypercone, $\mathbf{C}$, as above is \textbf{strictly minimising} if there is a $\Theta > 0$ such that for each $\varepsilon > 0$ we have 
    $$|\mathbf{C} \cap B_1(0)| \leq |S| - \Theta\varepsilon^n$$
    whenever $S$ is an integer rectifiable current with $\mathrm{Spt}(S) \subset \mathbb{R}^{n+1} \setminus B_\varepsilon(0)$ and $\partial S = \partial (\mathbf{C} \cap B_1(0))$.
    
\end{itemize}

\section{Stretching}\label{sec: stretching}

Our main result is the following:

\begin{theorem}\label{thm: stretching}
    If $\Sigma \subset M$ is a smooth, closed, connected, separating hypersurface which is uniquely homologically area-minimising in a tubular neighbourhood of $\Sigma$ in $(M,g)$, then there is a Riemannian metric, $h$, on $M$ such that $\Sigma$ bounds an isoperimetric region in $(M,h)$.
\end{theorem}
\begin{proof}
    We let $\Omega$ denote one of the two open sets bounded by $\Sigma$, and $V \subset M$ denote a tubular neighbourhood of $\Sigma$ in which it is uniquely homologically area-minimising. By taking a smaller neighbourhood of $\Sigma$ if necessary we may suppose that $V$ has smooth boundary, $\partial V = \Gamma^+ \cup \Gamma^-$, where each of the boundary components, $\Gamma^\pm$, are smooth, connected, and arise as the images of $\Sigma$ under the exponential map (i.e.~level sets of the signed distance function to $\Sigma$), chosen so that $\Gamma^- \subset \Omega$ and $\Gamma^+ \subset M \setminus \Omega$. In particular, we then ensure that $\Sigma$ generates the codimension-one homology of $V$ (with $\mathbb{Z}$ coefficients if $\Sigma$ is oriented, and $\mathbb{Z}_2$ coefficients otherwise).
    
    \bigskip
    
    We now alter the metric $g$ by stretching in the normal direction to $\Sigma$. For sufficiently small $\varepsilon > 0$ and for each $\Gamma = \Gamma^\pm$, we fix a collar, $\Phi: \Gamma \times[0,\varepsilon] \to V$, given by the exponential map on $\Gamma$ (with the unit normal on $\Gamma$ pointing into $V$) so that $\Phi(\Gamma \times [0,\varepsilon]) \cap \Sigma = \emptyset$. In this collar, the metric $g$ then has the form
    $$
        \Phi^*g = h_t + dt^2,
    $$
    where $\{h_t\}_{t \in [0,\varepsilon]}$ denotes a smoothly varying family of metrics on $\Gamma$.
      Now, we choose a smooth cutoff function, $\eta:[0,\varepsilon]\to[0,1]$,  which satisfies
    \[
    \begin{cases}
    \eta(t)=0 \ \text{for } t\in\left[0,\frac{\varepsilon}{10}\right]\cup\left[\frac{9\varepsilon}{10},\varepsilon\right],\\
    \eta(t)=1 \ \text{for } t\in\left[\frac{2\varepsilon}{10},\frac{8\varepsilon}{10}\right].
    \end{cases}
    \]
    Next, we choose a smooth metric, $h_\Gamma$, on $\Gamma$ such that $h_\Gamma \ge \max_{t\in[0,\varepsilon]} h_t$,
    and then define on $\Gamma \times [0,\varepsilon]$ the metric by setting
    $$
    \Phi^* \tilde g
    =
    \Bigl((1-\eta(t))\,h_t(x)+\eta(t)\,h_\Gamma(x)\Bigr)
    + dt^2.
    $$
    Note then that $\Phi^*\tilde g=\Phi^*g$ for $t\in[0,\varepsilon/10]\cup[9\varepsilon/10,\varepsilon]$, so that the resulting metric, $\tilde{g}$, is unchanged in a neighbourhood of $\Gamma^\pm$ and $\Sigma$. Moreover, we ensure that  $\tilde{g}$ is cylindrical on $\Phi(\Gamma\times     [\varepsilon/3,2\varepsilon/3])$ and such that $\tilde{g} \ge g$.
    Considering the two surfaces, $\Phi(\Gamma\times \{\frac{\varepsilon}{3}\})$ and $\Phi(\Gamma\times \{\frac{2\varepsilon}{3}\})$, we now alter $\tilde{g}$ and for each $R > 1$ construct metrics, $g_R$, so that $\mathrm{dist}_{g_R}(\Phi(\Gamma\times \{\frac{\varepsilon}{3}\}), (\Phi(\Gamma\times \{\frac{2\varepsilon}{3}\}))=R$.
    We choose another smooth cutoff function, $\zeta:[0,\varepsilon]\to[0,1]$,  which satisfies
    \[
    \begin{cases}
    \zeta(t)=0 \ \text{for } t\in\left[0,\frac{2\varepsilon}{10}\right]\cup\left[\frac{8\varepsilon}{10},\varepsilon\right],\\
    \zeta(t)=1 \ \text{for } t\in\left[\frac{3\varepsilon}{10},\frac{7\varepsilon}{10}\right].
    \end{cases}
    \]
    Denoting $\ell = \mathrm{dist}_{\tilde g}\left(\Phi\left(\Gamma\times\{\tfrac{\varepsilon}{3}\}\right),\ 
    \Phi\left(\Gamma\times\{\tfrac{2\varepsilon}{3}\}\right)\right)  = \frac{\varepsilon}{3}$, for each $R > 1$ we consider the stretching function
    \[
    \rho_R(t) = 1+\left((R/\ell)^2-1\right)\zeta(t),
    \]
    so that $\rho_R = 1$  on $[0,2\varepsilon/10]\cup
    [8\varepsilon/10,\varepsilon]$, and
    $\rho_R = (R/\ell)^2$ on $[\varepsilon/3,\,2\varepsilon/3]$. We then define the metrics $g_R$ for each $R > 1$ by setting:
    \[
    \begin{cases}
        g_R = \tilde{g} \qquad  &\text{ on } M \setminus \Phi(\Gamma\times[0,\varepsilon]),\\
        \Phi^*g_R = \tilde{g}_t + \rho_Rdt^2\qquad &\text{ on } \Gamma\times[0,\varepsilon].
    \end{cases}
    \]
    Thus, under this construction we ensure that $\mathrm{dist}_{g_R}(\Phi(\Gamma\times \{\frac{\varepsilon}{3}\}), (\Phi(\Gamma\times \{\frac{2\varepsilon}{3}\}))=R$ and $g_R$ remains cylindrical on $\Phi(\Gamma\times     [\varepsilon/3,2\varepsilon/3])$ for each $R > 1$ (c.f.~the right-hand graphic in Figure \ref{fig: stretching} for a depiction of $(M,g_R)$). By the above construction, we ensure the following properties:
    \begin{enumerate}
        \item $\Sigma \subset V$ is the unique homological minimiser in $(V,g_R)$  with $\mathrm{Per}_{g_R} (\Omega) = \mathrm{Per}_{g}(\Omega)$ for all $R>1$.
        \item $\mathrm{Vol}_{g_R}(\Omega), \mathrm{Vol}_{g_R}(M \setminus \Omega), \mathrm{Vol}_{g_R}(V)  \rightarrow \infty$ as $R \rightarrow \infty$.
        
        \item $\mathrm{Vol}_{g_R}(M \setminus V) = C$ for a constant $C > 0$ independent of $R > 1$.
       \end{enumerate}

    We now show that, when $R > 1$ is made sufficiently large, $\Omega$ is an isoperimetric region for its enclosed volume in $(M,g_R)$. First, noting that for each $R > 1$ the metrics $g_R$ have the same uniform bounds on sectional curvature and their injectivity radius, by \cite[Lemma 4.3]{N24} (see also \cite[Lemma 2]{MNP25}) we ensure a uniform bound on the (constant) mean curvature of the smooth portions of any isoperimetric boundary enclosing the same volume as $\Omega$ with respect to the metric $g_R$. Therefore, as in the proof of \cite[Lemma 4.4]{N24} (see also \cite[Lemma 2]{C91} and the remark after its proof), the monotonicity formula guarantees that there is some uniform $\delta > 0$, independent of $R > 1$, so that each connected boundary component has area at least $\delta$. Using this, we ensure uniform upper bounds, $D,L > 0$, independent of $R > 1$ on the extrinsic diameter of each boundary component and the number of connected boundary components respectively (by direct comparison with the area of $\Sigma$). We now argue that, for $R > 1$ sufficiently large, there is only one boundary component for each such isoperimetric region, and that it must coincide with $\Sigma$.

    \bigskip
    
    Suppose now that $E$ is an isoperimetric region with the same volume as $\Omega$ in $(M,g_R)$ for some $R > 1$; from the above discussion that there are at most $L$ connected boundary components and each of these components has extrinsic diameter at most $D$. Note that if some connected boundary component, $S$, of $E$ both lies inside of $V$ and is homologous to $\Sigma$ in $V$, then since $\Sigma$ is uniquely homologically area-minimising in $V$ by property 1 above, we ensure that $|S|_{g_R} = |\Sigma|_{g_R}$ and thus $\partial E = \Sigma$ 
    (else we would have $\mathrm{Per}_{g_R}(E) > \mathrm{Per}_{g_R}(\Omega) + \delta$ by the uniform lower area bound on each boundary component). We thus can argue that some boundary component of $E$ both lies inside of $V$ and is homologous to $\Sigma$ in $V$ in order to conclude. 
    
    \bigskip

    Assuming then that $E$ has no boundary components that are both contained in $V$ and homologous to $\Sigma$ in $V$, we will derive a contradiction for $R > 1$ sufficiently large. Let us write the connected components of $\partial E$ as a disjoint union
    $$
    \partial E = S_1 \cup \dots \cup S_l
    $$
    for some integer $1 \leq l \leq L$. Defining the set of points at most distance $D$ from $M \setminus V$ as
    $$
    K_R = \{x \in M \, | \, \mathrm{dist}_{g_R}(x,M \setminus V) \leq D\},
    $$
    we see that if some $S_i \cap (M \setminus V) \neq \emptyset$ then the uniform diameter bound above implies that $S_i \subset K_R$. We note also that, by the uniform bounds on the geometry, we ensure that $\mathrm{Vol}_{g_R}(K_R) \leq C_0$ for a constant $C_0 > 0$ independent of $R > 1$.

    \bigskip
    Denoting $\mathcal{I} = \{i \in \{1,\dots,l\} \, | \, S_i \subset V \}$ for the indices of the boundary components of $E$ that lie entirely in $V$, since $\Sigma$ generates the codimension-one homology of $V$ we see that by our assumption, $S_i$ is null-homologous in $V$ for each $i \in \mathcal{I}$. Hence we can find some Caccioppoli set $U_i \subset \subset V$ such that $\partial U_i = S_i$ for each $i \in \mathcal{I}$; here we choose the set $U_i$ which does not intersect $\partial V$, which we can ensure since $S_i$ is null-homologous in $V$. By the uniform diameter bound above and the cylindrical structure of $V$, for each $i \in \mathcal{I}$ we ensure that $\mathrm{Vol}_{g_R}(U_i) \leq C_1$ for a constant $C_1 > 0$ independent of $R > 1$.

    \bigskip

    Setting $U = \Delta_{i \in \mathcal{I}} U_i$ to be the symmetric difference of the $U_i$ for $i \in \mathcal{I}$ and $F = E \Delta U$ (so that in particular $E = F \Delta U$), by considering the Caccioppoli sets as integral $(n+1)$-currents with coefficients in $\mathbb{Z}_2$, we have
    $$
    [F] = [E] + \sum_{i \in \mathcal{I}} [U_i],
    $$
    and thus by taking boundaries we see that
    $$
    \partial [F] = \partial [E] + \sum_{i \in \mathcal{I}} \partial [U_i] = \sum_{i = 1}^l [S_i] + \sum_{i \in \mathcal{I}} [S_i] = \sum_{i \notin \mathcal{I}} [S_i].
    $$
    Now, for $i \notin \mathcal{I}$ we have $S_i \cap (M \setminus V) \neq \emptyset$ and by the uniform diameter bound above we see that $S_i \subset K_R$, thus
    $$
    \partial F \subset \bigcup_{i \notin I} S_i \subset K_R;
    $$
    in other words, $F$ has no boundary components entirely in $V$. 

    \bigskip

    Now, since for $R > 1$ sufficiently large we have that $M \setminus K_R$ is connected, the constancy theorem (or the fact that then $F$ is relatively closed and open in $M \setminus K_R$) imply that either 
    $$
    F\subset K_R
    \qquad\text{or}\qquad
    M\setminus F\subset K_R,
    $$
    and since $E=F\triangle U$ as noted above, it follows that either
    $$
    E\subset K_R \cup\bigcup_{i\in\mathcal I}U_i
    \quad  \text{ or } \quad
    M\setminus E\subset K_R \cup\bigcup_{i\in\mathcal I}U_i.
    $$
    With the uniform volume control on both $K_R$ and the $U_i$ above, and the fact that $l \leq L$ independently of $R > 1$ we see that
    $$
    \min\left\{ \mathrm{Vol}_{g_R}(E), \mathrm{Vol}_{g_R}(M\setminus E) \right\} \leq \mathrm{Vol}_{g_R}(K_R) + \sum_{i\in\mathcal I}\mathrm{Vol}_{g_R}(U_i) \leq C_0+LC_1,
    $$ 
    independently of $R > 1$. However, this uniform volume bound contradicts property 2 of the construction above for sufficiently large $R > 1$ since both $\mathrm{Vol}_{g_R}(E) = \mathrm{Vol}_{g_R}(\Omega)$ and $\mathrm{Vol}_{g_R}(M \setminus E) = \mathrm{Vol}_{g_R}(M \setminus \Omega)$. Thus, we see that $E$ has some boundary component in $V$ which is also homologous to $\Sigma$ in $V$. As mentioned, property 1 above then ensures that $\partial E = \Sigma$, and thus $\Sigma$ bounds an isoperimetric region of enclosed volume $\mathrm{Vol}_{g_R}(\Omega)$ as desired.
\end{proof}
 
The proof of Theorem \ref{thm: stretching} relied on smoothness of the separating hypersurface only to ensure some topological control on the neighbourhood in which it is uniquely homologically minimising. Provided we have this control, we can carry out the same construction in the singular setting: 

\begin{corollary}\label{cor: singular stretching}
    If $\Sigma \subset M$ is a closed, connected, separating hypersurface which is uniquely homologically area-minimising in a neighbourhood, $V$, of $\Sigma$ in $(M,g)$, with $V$ satisfying the following topological assumptions: 

    \begin{itemize}
        \item $V$ is a cobordism between smooth connected $\Gamma^\pm \subset \Omega^\pm$ with $M \setminus \Sigma = \Omega^+ \cup \Omega^-$, where $\Omega^\pm$ are disjoint, $\Sigma$ separates $V$, and $\partial V$ separates $M$ into three connected components.

        \item $\Sigma$ generates the codimension-one homology of $V$ (with $\mathbb{Z}$ coefficients if $\mathrm{Reg}(\Sigma)$ is oriented, and $\mathbb{Z}_2$ coefficients otherwise).
    \end{itemize}
    Then, there is a Riemannian metric, $h$, on $M$ such that $\Sigma$ bounds an isoperimetric region in $(M,h)$.
\end{corollary}

\begin{proof}
    After choosing a neighbourhood, $V$, of $\Sigma$ satisfying the topological assumptions, one can proceed to argue identically as in the proof of Theorem \ref{thm: stretching} to conclude.
\end{proof}

\begin{remark}
    As utilised in Subsection \ref{subsec: singular sets} below, the examples of singular minimal hypersurfaces constructed in \cite{S00} have neighbourhoods satisfying the above topological assumptions of Corollary \ref{cor: singular stretching}.
\end{remark}

The method of proof of Theorem \ref{thm: stretching} also allows for the construction of isoperimetric regions from volume-constrained minimisers that are, a priori, close in area to a local area-minimiser:

\begin{corollary}\label{cor: stretching for vcm}  
    Suppose that $\Sigma,\widetilde{\Sigma} \subset M$ are smooth, closed, connected, separating hypersurfaces satisfying the following assumptions:
    
    \begin{itemize}
        \item $\Sigma$ is homologically area-minimising in a tubular neighbourhood, $V$, of $\Sigma$ in $(M,g)$.

        \item $\widetilde{\Sigma}$ bounds a unique volume-constrained minimiser in $V$ with respect to the metric $g$ (up to taking complements).

        \item $\widetilde{\Sigma}$ is both homologous to $\Sigma$ and sufficiently close in area to $\Sigma$ with respect to the metric $g$. 
    \end{itemize}
    Then, there is a Riemannian metric, $h$, on $M$ such that $\widetilde{\Sigma}$ bounds an isoperimetric region in $(M,h)$.
\end{corollary}

\begin{proof}
    Denoting $\widetilde{\Omega} \subset M$ an open set bounded by $\widetilde{\Sigma}$, one can follow the construction in the proof of Theorem \ref{thm: stretching} verbatim to ensure that, for $R > 1$ sufficiently large, any isoperimetric region, $E \subset M$, of enclosed volume $\mathrm{Vol}_{g_R}(\widetilde{\Omega})$ in $(M,g_R)$ has at least one connected boundary component, say $S \subset \partial E$, homologous to $\Sigma$ in $V$.
    
    \bigskip
    
    Under the assumption that $\widetilde{\Sigma}$ is sufficiently close in area to $\Sigma$, namely provided
    $\mathrm{Per}_{g_R}(\widetilde{\Omega}) < \mathrm{Per}_{g_R}(\Omega) + \delta$ (where $\partial \Omega = \Sigma$ and $\delta > 0$ is the uniform lower area bound as in the proof of Theorem \ref{thm: stretching}), we now deduce that $E$ must have exactly one boundary component (i.e.~that $\partial E = S$). To see this, since $\Sigma$ is locally homologically area-minimising in $V$ we have $|S|_{g_R} \geq \mathrm{Per}_{g_R}(\Omega)$, and as $E$ is isoperimetric we have 
    $$\mathrm{Per}_{g_R}(\Omega) \leq |S|_{g_R} \leq \mathrm{Per}_{g_R}(E) \leq \mathrm{Per}_{g_R}(\widetilde{\Omega}) < \mathrm{Per}_{g_R}(\Omega) + \delta;$$
    hence $E$ has exactly one boundary component as any additional boundary components would contribute at least $\delta$ perimeter, forcing $\mathrm{Per}_{g_R}(E) \geq \mathrm{Per}_{g_R}(\Omega) + \delta$, a contradiction. The volume-constrained minimisation of $\widetilde{\Omega}$ in $V$ then ensures that $\mathrm{Per}_{g_R}(\widetilde{\Omega}) = \mathrm{Per}_{g_R}(E)$, and the uniqueness implies that $\partial E = \widetilde{\Sigma}$ as desired.
\end{proof}

\begin{remark}
    With the same assumptions and reasoning as for Corollary \ref{cor: singular stretching}, Corollary \ref{cor: stretching for vcm} can, in principle, be applied to the boundaries of volume-constrained minimisers with non-empty singular sets.
\end{remark}

\section{Applications}\label{sec: applications}

We record some direct applications of Theorem \ref{thm: stretching}, showing that the space of isoperimetric boundaries in $M$ is rich both in terms of topological complexity and singular behaviour.

\subsection{Topological types}\label{subsec: topology}

We show that the space of isoperimetric boundaries is topologically diverse:

\begin{theorem}\label{thm: topological constructions}
    Given a smooth, closed, connected, separating hypersurface, $\Sigma \subset M$, there is a Riemannian metric, $g$, on $M$ such that $\Sigma$ bounds an isoperimetric region in $(M,g)$.
\end{theorem}

\begin{proof}
    Starting with an arbitrary Riemannian metric, $g$, the idea is to find a suitable metric perturbation which ensures $\Sigma$ is uniquely homologically area-minimising in a tubular neighbourhood, then directly apply Theorem \ref{thm: stretching} to obtain the desired conclusion.
    
    \bigskip

    For a smooth function, $f$, on $M,$ if we define $h = e^{2f}g$ then by \cite[1.159/(1.163)]{B87} the mean curvature, $H_{h}$, of $\Sigma$ with respect to $h$ satisfies
    $$H_{h} = e^{-f}\left(H_g - \frac{\partial f}{\partial \nu}\right),$$
    and the Ricci curvature, $\mathrm{Ric}_{h}$, of $M$ with respect to $h$ satisfies
    $$\mathrm{Ric}_{h} = \mathrm{Ric}_g - (n-1)(\nabla^2_gf - df \otimes df) - (\Delta_g f + (n-1)|\nabla_gf|^2)g.$$ 
    Choosing a smooth function, $f_1$, on $M$ with $\frac{\partial f_1}{\partial \nu} = H_g$, we ensure that $\Sigma$ is a minimal hypersurface in $(M,h)$. We then choose another smooth function, $f_2$, on $M$ so that both $f_2$ and $\nabla^{h}f_2$ vanish on $\Sigma$, ensuring that $\nu_{h}$ is also a unit normal with respect to the metric $\tilde{h} = e^{2f_2}h$, $\sff_{\Sigma,h} = \sff_{\Sigma,\tilde{h}}$ (equality of the second fundamental forms of $\Sigma$ with respect to each metric), and $\Delta_{h}f_2 = \nabla^2_{h}f_2(\nu,\nu)$ on $\Sigma$. We thus compute that
    $$\mathrm{Ric}_{\tilde{h}}(\nu_h,\nu_h) = \mathrm{Ric}_h(\nu_h,\nu_h) - n\nabla^2_hf_2(\nu_h,\nu_h),$$
    and so choosing $\nabla^2_hf_2(\nu_h,\nu_h)$ sufficiently large (in particular so that $\mathrm{Ric}_{\tilde{h}}(\nu_h,\nu_h) < - |\sff_{\Sigma,h}|^2$ on $\Sigma$), we ensure that $\Sigma$ is a strictly stable minimal hypersurface in $(M,\tilde{h})$ in the sense that for each non-zero $\phi \in C^{\infty}_c(\Sigma)$ we have
    $$\int_{\Sigma} |\nabla^{\tilde{h}} \phi|^2 - (|\sff_{\Sigma,\tilde{h}}|^2 +\mathrm{Ric}_{\tilde{h}}(\nu_h,\nu_h))\phi^2 \, dV_{\tilde{h}} > 0.$$
    Concretely, one can choose $f_2$ so that it agrees with a large multiple of the square of the distance function to $\Sigma$ in a tubular neighbourhood.
    By the results of \cite{W94}, the strict stability of $\Sigma$ ensures its unique homological area-minimisation in a tubular neighbourhood. Applying Theorem \ref{thm: stretching} in this tubular neighbourhood, we find a Riemannian metric on $M$ for which $\Sigma$ bounds an isoperimetric region as desired.
\end{proof}

\begin{remark}
    While above we prescribe the ambient closed manifold and ensure isoperimetry for any smooth, closed, connected, separating hypersurface, the same technique shows that one can instead prescribe any null-cobordant, smooth, closed, connected manifold, $\Sigma$, and construct an ambient closed manifold (formed by doubling the manifold bounded by $\Sigma$) along with a Riemannian metric on it so that $\Sigma$ bounds an isoperimetric region.
\end{remark}

\begin{remark}
    From the proof of Theorem \ref{thm: topological constructions}, the resulting isoperimetric boundaries, $\Sigma$, are strictly stable and locally uniquely homologically area-minimising. One can instead produce isoperimetric boundaries with small non-zero constant mean curvature by applying Corollary \ref{cor: stretching for vcm}. More precisely, the strict stability of $\Sigma$ ensures the invertibility of its Jacobi operator and hence, for small constant mean curvature values, one can find strictly stable CMC graphs, $\widetilde{\Sigma}$, over $\Sigma$. The results of \cite{W94} (see also \cite{G96}) ensure that such $\widetilde{\Sigma}$ are locally uniquely volume-constrained minimising and hence, as the graphs will have area close to that $\Sigma$ for small values of the mean curvature, one can apply Corollary \ref{cor: stretching for vcm} to ensure that $\widetilde{\Sigma}$ bounds an isoperimetric region for some Riemannian metric on $M$.
\end{remark}

\subsection{Singular sets}\label{subsec: singular sets}

Our initial motivation for this work stemmed from a claim made in \cite{M16}, concerning the possibility of constructing singular isoperimetric regions in Riemannian spheres, which we now describe. If we consider a strictly stable, strictly minimising, minimal hypercone, $\mathbf{C} \subset \mathbb{R}^{n+1}$, with $\mathrm{Sing}(\mathbf{C}) = \{0\}$ (thus necessarily $n + 1\geq 8$), we see that  $\Sigma = (\mathbf{C} \times \mathbb{R}) \cap \mathbb{S}^{n+1}$ is a minimal hypersurface in $(\mathbb{S}^{n+1},g_{\mathrm{round}})$ with two isolated singularities (antipodal to one another), where $g_\mathrm{round}$ is the metric on $\mathbb{S}^{n+1}$ induced from $g_\mathrm{eucl}$. The results of \cite[Section 3]{S99} show that, after successive conformal metric perturbations, $\Sigma$ is in fact uniquely homologically area-minimising in a neighbourhood with respect to some Riemannian metric on $\mathbb{S}^{n+1}$. It was then claimed in \cite[Section 13.2]{M16} that ``if you make
the metric huge on the rest of the sphere, it should be globally area-minimising for its volume'' which we interpreted to mean that $\Sigma$ can be shown to bound an isoperimetric region for some choice of Riemannian metric on $\mathbb{S}^{n+1}$. The discussion on \cite[Pages 168-169]{S99} shows that the hypotheses of our Corollary \ref{cor: singular stretching} are satisfied, and hence its application directly verifies the above claim.

\bigskip

In fact, we can use our Corollary \ref{cor: singular stretching} to produce isoperimetric boundaries with the singular sets exhibited in \cite{S00}. There, a variety of singular minimal hypersurfaces were constructed, and then shown to be locally area-minimising, in closed Riemannian manifolds. We first introduce the class of manifolds from which the singular sets will be formed:

\begin{definition}
    For $n\ge 7$, we denote by $\mathcal{F}(n+1)$ the collection of smooth, closed, connected, orientable manifolds, $\Gamma$, of dimension at most $n - 7$ such that there exists a smooth embedding $\Phi: \overline{B}^{m+1}_1(0) \times \Gamma \to \mathbb{R}^{n+1}$, where $m + \mathrm{dim}( \Gamma) + 1 = n + 1$ and $\overline{B}^{m+1}_1(0)$ denotes the closed unit ball in $\mathbb{R}^{m+1}$.
\end{definition}
\begin{remark}
    As noted in \cite{S00}, there are many examples of manifolds in the class $\mathcal{F}(n+1)$. For instance, if $\Gamma $ is embedded into $\mathbb{R}^{n+1}$ with trivial normal bundle, then $\Gamma \in \mathcal{F}(n+1)$. In particular, if $\Gamma$ is embedded in $\mathbb{R}^{\mathrm{dim} (\Gamma) +1}$ with $\mathrm{dim}(\Gamma)\le n-7$, then $\Gamma $ is also in $\mathcal{F}(n+1)$.
\end{remark}

With the above notation, we construct singular isoperimetric boundaries as follows:

\begin{theorem}\label{thm: singular constructions}
    Suppose that $M$ is an orientable, closed manifold of dimension $n+1\ge 8$ and $L \geq 1$, then for any collection of Riemannian manifolds, $\{(\Gamma_i, h_i)\}_{i=1}^L$,
    with $\Gamma_i\in \mathcal{F}(n+1)$, and any collection, $\{\mathbf{C}_i\}_{i = 1}^L$, of strictly stable, strictly minimising, regular minimal hypercones, $\mathbf{C}_i\subset \mathbb{R}^{n_i+1}$, where we denote $n_i=n-\mathrm{dim} (\Gamma_i)$, there is a Riemannian metric, $g$, on $M$, and an isoperimetric region, $\Omega$, in $(M,g)$ with boundary $\Sigma$ such that
    \[
    \mathrm{Sing}(\Sigma)=\bigcup_{i=1}^L \Lambda_i,
    \]
    where $\Lambda_i$ is an embedded copy of $\Gamma_i$ for $i=1,\dots,L$. In particular, in a neighbourhood of $\Lambda_i$ for each $i=1,\dots,L$, $(M,g)$ is locally isometric to $B^{n_i+1}_\sigma(0)\times \Gamma_i$
    endowed with the product metric $g_{\mathrm{eucl}}\times h_i$, 
    with $\Sigma$ in this region isometric to $(\mathbf{C}_i \cap B^{n_i+1}_\sigma(0))\times \Gamma_i$, for some $\sigma > 0$ sufficiently small.
\end{theorem}

\begin{proof}
First, we pick a smooth, closed, connected, separating hypersurface, $S\subset M$; for example, one can choose $S$ to be the boundary of a small geodesic ball for some Riemannian metric on $M$. Because $M$ is orientable, $S$ is also orientable. We now describe the procedure used in the proof of \cite[Theorem 1]{S00} in order to construct singular minimal hypersurfaces from $S$. Consider a tubular neighbourhood, $\mathcal{E}$, which $S$ separates into disjoint connected components, $\mathcal{E}^\pm$, then choose points, $p_1,\dots,p_L\in S$, and pairwise disjoint balls, $B_i$, each diffeomorphic to a Euclidean ball of dimension $n+1$ centred at each point, $p_i$, so that $B_i \subset \subset \mathcal{E}$. Then, for each $i = 1, \dots, L$, the ball $B_i$ is separated by $S$ into two disjoint connected components, $B_i^\pm$.

\bigskip
For each $i=1,\dots,L$, let $ \mathbf{\widehat C}_i\subset B^{n_i+1}_1(0)$ be a capped-off version of the minimising stable hypercone $\mathbf{C}_i$ as constructed in \cite[Section 2, Proposition]{S00}; more precisely, $\mathbf{\widehat C}_i\subset B^{n_i+1}_1(0)$ is a closed hypersurface, smoothly embedded away from the origin, and such that $\mathbf{\widehat C}_i \cap B^{n_i+1}_\sigma(0) = \mathbf{C}_i \cap B^{n_i+1}_\sigma(0)$ for some $\sigma > 0$. By the hypothesis on $\Gamma_i$ and since each $B_i$ is diffeomorphic to a Euclidean ball of dimension $n+1$, for each $i = 1,\dots,L$ we choose an embedding
\[
\Phi_i: B^{n_i+1}_1(0)\times \Gamma_i \hookrightarrow B_i^+,
\]
and define
\[
\Sigma_i = \Phi_i(\mathbf{\widehat  C}_i\times \Gamma_i)\subset B_i^+ .
\]
Next, for each $i = 1,\dots,L$ we choose two $n$-disks, $D_i \subset \Sigma_i,D_i'\subset S\cap B_i$, which by removing $D_i$ and $D_i'$, allow us to connect $S$ and $\Sigma_i$ inside $B_i^+$ by an $n$-dimensional handle (a hypersurface diffeomorphic to $S^{n-1}\times I$) whose boundary components are glued to $\partial D_i$ and $\partial D_i'$.
We then finally let $\Sigma$ be the resulting smooth, closed, connected hypersurface
\[
\Sigma \;=\; S \# \Sigma_1 \# \cdots \# \Sigma_L,
\]
where $\#$ denotes the connected sum, which can be done in such a way to ensure that $\Sigma$ is embedded away from $\mathrm{Sing}(\Sigma)$, is orientable, and separates $M$. For each $i= 1,\dots, L$, on the images, $\Phi_i(B^{n_i+1}_1(0)\times \Gamma_i)$, we define the pull-back metric
\[
g_i = (\Phi_i^{-1})^*(g_{eucl}\oplus h_i).
\]
Then, by using a partition of unity, we can choose a Riemannian metric, $g_0$, on $M$ such that it agrees with $g_i$ on $\Phi_i\bigl(B^{n_i+1}_{1/2}(0) \times \Gamma_i\bigr)$ for each $i=1,\dots,L$.
The hypersurface $\Sigma \subset (M,g_0)$ then satisfies each of the assumptions of \cite[Lemma 1]{S00}, and hence we can apply it to find a Riemannian metric, $h$, on $M$
that agrees with $g_0$ in a neighbourhood of $\mathrm{Sing}(\Sigma)$ and ensures that $\Sigma$ is uniquely homologically minimising in a tubular neighbourhood. 

\bigskip

As noted on \cite[Pages 2323-2324]{S00} (see also the discussion on \cite[Pages 168-169]{S99}), we can find a tubular neighbourhood of $\Sigma$ in which the assumptions of Corollary \ref{cor: singular stretching} are satisfied by virtue of the above construction. We can thus apply Corollary \ref{cor: singular stretching} in order to conclude that $\Sigma$ bounds an isoperimetric region for some Riemannian metric on $M$ as desired.
\end{proof}

\begin{remark}\label{rem: cones}
    As mentioned in Section \ref{sec: introduction}, due to a topological restriction, the singular examples of isoperimetric boundaries constructed in \cite{N24} required the prescribed regular minimal hypercones, $\mathbf{C}_i$, on which the singularities were modelled to be symmetric Simons cones, namely those with link given by the products of spheres, $\mathbb{S}^p(r)\times \mathbb{S}^p(r)\subset \mathbb{S}^{2p+1}$ for $p\ge3$ and some $r>0$. Our construction does not need this constraint, and ensures the singularities have unique cylindrical tangent cones, of the form $\mathbf{C}\times \mathbb{R}^k$, for any choice, $\mathbf{C} \subset \mathbb{R}^{n - k + 1}$, of strictly stable, strictly minimising regular minimal hypercone; this directly addresses a question left open by the constructions in \cite{N24}.
\end{remark}

\begin{remark}
    The results of \cite{MNP25} show that isoperimetric boundaries are generically (in various settings) smooth and nondegenerate in closed manifolds of dimension eight. In particular, \cite[Theorem 1]{MNP25} ensures that the above examples of Riemannian metric and enclosed volume pairs on a given closed eight-manifold which admit singular isoperimetric regions (having only finitely many isolated singularities in this case) are necessarily contained in a Baire meagre set of such pairs. Moreover, the existence of singular isoperimetric boundaries in any closed manifold of dimension eight, as guaranteed by Theorem \ref{thm: singular constructions}, retroactively shows that the results of \cite{MNP25} are sharp.
\end{remark}

\bibliographystyle{alpha}
\bibliography{main}

\hrule
\Addresses
\end{document}